\documentclass[12pt,reqno]{article}

\usepackage[usenames]{color}
\usepackage{amssymb}
\usepackage{graphicx}
\usepackage{amscd}

\usepackage[colorlinks=true,
linkcolor=webgreen,
filecolor=webbrown,
citecolor=webgreen]{hyperref}

\definecolor{webgreen}{rgb}{0,.5,0}
\definecolor{webbrown}{rgb}{.6,0,0}

\usepackage{color}
\usepackage{fullpage}
\usepackage{float}
\usepackage{amsmath}
\usepackage{amsthm}
\usepackage{amsfonts}
\usepackage{latexsym}

\begin{document}

\theoremstyle{plain}
\newtheorem{theorem}{Theorem}
\newtheorem{example}[theorem]{Example}
\newtheorem{corollary}[theorem]{Corollary}
\newtheorem{remark}[theorem]{Remark}
\newtheorem{lemma}[theorem]{Lemma}

\begin{center}
\vskip 1cm{\LARGE\bf
Additional Fibonacci--Bernoulli relations\\
\vskip.1in}\vskip 0.8cm\large
Kunle Adegoke\\
Obafemi Awolowo University\\
Ile-Ife, Nigeria\\
\href{mailto:adegoke00@gmail.com}{\tt adegoke00@gmail.com} \\
\vskip .4cm\large
Robert Frontczak\footnote{Statements and conclusions made in this article are entirely those of the author. They do not necessarily reflect the views of LBBW.} \\
Landesbank Baden-W\"urttemberg\\
Stuttgart,  Germany\\
\href{mailto:robert.frontczak@lbbw.de}{\tt robert.frontczak@lbbw.de} \\
\vskip .4cm\large
Taras Goy\\
Vasyl Stefanyk Precarpathian National University\\
Ivano-Frankivsk, Ukraine\\
\href{mailto:taras.goy@pnu.edu.ua}{\tt taras.goy@pnu.edu.ua}
\end{center}

\vskip .15 in

\begin{abstract} We continue our study on relationships between Fibonacci (Lucas) numbers and Bernoulli numbers and polynomials. The derivations of our results are based on functional equations for the respective generating functions, which in our case are combinations of hyperbolic functions. Special cases and some corollaries will highlight interesting aspects of our 	findings.
\end{abstract}

\section{Introduction}

Fibonacci and Lucas numbers satisfy the linear second-order  recurrence 
\begin{equation*}
u_{n} = u_{n-1} + u_{n-2},\quad n\geq 2,
\end{equation*} 
with initial conditions $F_0 = 0$, $F_1 = 1$ and $L_0 = 2$, $L_1 = 1$, respectively. Both sequences have a long history and are very popular among mathematicians as they appear in important mathematical branches such as number theory, combinatorics and graph theory. They have entries A000045 and A000032 in the On-Line Encyclopedia of Integer Sequences \cite{Sloane}. Excellent references on these sequences are the books \cite{Koshy,Vajda}. 

As usual, Bernoulli polynomials  are defined by the exponential generating function  \cite[Chapter 1, Section 1.3]{Jeff}
\begin{equation*}
H(x,z):= \sum_{n=0}^\infty B_n(x) \frac{z^n}{n!} = \frac{z e^{xz}}{e^z - 1}, \qquad |z|<2\pi. 
\end{equation*} 

For $n\geq0$, they satisfy the following functional relations 
\cite[Chapter 1, Section 1.3]{Jeff}:
\begin{equation}\label{1FuncRel-B}
B_n (1 + x)  - B_n (x) = nx^{n - 1},  
\end{equation}
\begin{equation}\label{2FuncRel-B}
B_n (1 - x) = ( - 1)^n B_n (x),
\end{equation}
\begin{equation}\label{3FuncRel-B}
B_n ( - x) = ( - 1)^n\big(B_n (x) + nx^{n - 1}\big),
\end{equation}
\begin{equation}\label{Raabe-B}   
{B_n (mx)} = {m^{n - 1}}\sum_{k = 0}^{m - 1} {B_n\Big( {x + \frac{k}{m}} \Bigr)},\qquad m \geq1.
\end{equation}

Also, Bernoulli polynomials have the property
\begin{equation*}
B_n (x + z) = \sum_{k = 0}^n {\binom nkB_{n - k} (x)z^{k} },
\end{equation*}
from which we get
\begin{equation}\label{eq.qr64rlf}
B_n (x) = \sum_{k = 0}^n {\binom nk B_{n - k} x^{k} },
\end{equation}
where $B_n=B_n(0)$ are the Bernoulli numbers. The Bernoulli numbers are rational numbers starting with 
\begin{equation*}
B_0=1,\,\,\,\, B_1=-\frac12, \,\,\,\, B_2=\frac16, \,\,\,\, B_4=-\frac1{30}, \,\,\,\, B_6=\frac1{42},
\end{equation*}
and $B_{2n+1}=0$ for $n\geq1$. 

The following identities connect Fibonacci numbers to Bernoulli polynomials and are proved in \cite{RFTG1}: 
for each integers $n\geq 0$, $j\geq 1$, and $x\in\mathbb{C}$,
\begin{equation*}\label{FBpol1}
\sum_{k=0}^n {n\choose k} F_{jk} \big(\sqrt{5}F_j\big)^{n-k} B_{n-k}(x) = n F_j \big ((\sqrt{5}x+\beta)F_j + F_{j-1}\big )^{n-1},
\end{equation*}
\begin{equation*}\label{FBpol2}
\sum_{k=0}^n {n\choose k} F_{jk} \big(-\sqrt{5}F_j\big)^{n-k} B_{n-k}(x) = n F_j \big ((\alpha - \sqrt{5}x)F_j + F_{j-1}\big )^{n-1},
\end{equation*}
where $\alpha=\frac{1+\sqrt{5}}2$ is the golden ratio and $\beta=\frac{1-\sqrt{5}}{2}=-\frac{1}{\alpha}$. Other results in this direction are contained in \cite{Byrd,Castellanos,RF-JIS,RF-Tomovski,GuoChu,Young,Zhang-Ma,Zhang,Zhang-Guo}, among others.

In this paper, we state new relations involving Fibonacci and Lucas numbers and Bernoulli numbers and polynomials. We will work with many exponential generating functions. The results stated are complements of the recent discoveries from \cite{RF-JIS,RFTG1,RFTG2}. Some of our results were announced without proofs in \cite{adegoke_Conf}.

\section{Identities from hidden threefold convolutions}

We begin with a known lemma \cite[Vol.~1, p.~251]{Koshy}.
\begin{lemma}
	\label{lem1}
	Let $n$ and $j$ be positive integers. Then
	\begin{equation}\label{eqlem1}
	\sum_{k=0}^{n} {n\choose k} F_{jk} F_{j(n-k)} = \frac{ 2^n L_{jn} - 2 L_j^n}{5},
	\end{equation}
	\begin{equation}\label{eqlem2}
	\sum_{k=0}^{n} {n\choose k} L_{jk} L_{j(n-k)} = 2^n L_{jn} + 2 L_j^n,
	\end{equation}
	\begin{equation}\label{eqlem3}
	\sum_{k=0}^{n} {n\choose k} F_{jk} L_{j(n-k)} = 2^n F_{jn}.
	\end{equation}
\end{lemma}

Our first main result is the following theorem. 
\begin{theorem} \label{thm1} Let $n$ and $j$ be positive integers. Then
	\begin{equation} \label{thm1-1}
	\sum_{\substack{k=0 \\ n - k \equiv 0 \,(\mathrm {mod} \,2)}}^n \binom{n}{k} 
	\big ( 2^k L_{jk} - 2 L_j^k \big ) (\sqrt{5} F_j)^{n-k} \frac{B_{n-k+2}}{n-k+2}= \frac{2^{n+2}L_{j(n+2)}-2L_j^{n+2}}{5(n+1)(n+2)F_j^2} - L_j^n, 
	\end{equation}
	\begin{equation} \label{thm1-2}
	\sum_{\substack{k=0 \\ n - k \equiv 0 \,(\mathrm {mod} \,2)}}^n \binom{n}{k} 
	\big ( 2^k L_{jk} + 2 L_j^k \big ) (\sqrt{5} F_j)^{n-k} \frac{2^{n-k+2}-1}{n-k+2} B_{n-k+2} = L_j^n,
	\end{equation}
	and
	\begin{align}\label{thm1-3}
	\sum_{\substack{k=0 \\ n - k \equiv 0 \,(\mathrm {mod} \,2)}}^n\!\!\binom{n}{k} 2^k F_{jk}& (\sqrt{5} F_j)^{n-k} \frac{B_{n-k+2}}{n-k+2}+ \frac{2}{\sqrt{5}} \sum_{\substack{k=0 \\ n - k \equiv 1 \,(\mathrm {mod} \,2)}}^{n-1}\!\!\binom{n}{k} L^k_{j} (\sqrt{5} F_j)^{n-k} 
	\frac{B_{n-k+1}}{n-k+1}\notag\\
	&= \frac{2^{n+3}F_{j(n+2)}}{5(n+1)(n+2)F_j^2} - \frac{2L_j^{n+1}}{5(n+1)F_j}.
	\end{align}
\end{theorem}
\begin{proof} Let $F(z)$ and $L(z)$ denote the exponential generating functions of sequences $(F_{jn})_{n\geq 0}$ and $(L_{jn})_{n\geq 0}$, respectively, with  $j\geq 1$. Then, it is easy to derive
	\begin{equation}\label{egF}
	F^2(z) = \frac{4}{5} e^{L_j z}\sinh^2 \Big ( \frac{\sqrt{5}F_j}{2}  z\Big), 
	\end{equation}
	\begin{equation}\label{egL}
	L^2(z) = 4 e^{L_j z}\cosh^2 \Big ( \frac{\sqrt{5}F_j}{2}  z\Big ).
	\end{equation}
	
	From the power series for the cotangent \cite[Chapter 1.3]{Jeff}
	\begin{equation*}
	\coth z = \sum_{n=0}^\infty 2^{2n} B_{2n} \frac{z^{2n-1}}{(2n)!}
	\end{equation*}
	we get
	\begin{equation}\label{cot}
	-\frac{d}{dz} \coth z = \frac{1}{\sinh^2 z} = \frac{1}{z^2} - \sum_{n=1}^\infty (2n-1) 2^{2n} B_{2n} \frac{z^{2n-2}}{(2n)!}\,.
	\end{equation}
	
	Hence, we see that the functional equation \eqref{egF}, using \eqref{eqlem1} and \eqref{cot}, can be written equivalently as \begin{equation*}
	e^{L_j z} = S_1(z) - S_2(z)
	\end{equation*}
	with
	\begin{equation*}
	S_1(z) = \frac{1}{5F_j^2} \sum_{n=0}^\infty \big ( 2^n L_{jn} - 2 L_j^n \big ) \frac{z^{n-2}}{n!}\,,
	\end{equation*}
	\begin{equation*}
	S_2(z) = \sum_{n=0}^\infty \big ( 2^n L_{jn} - 2 L_j^n \big ) \frac{z^{n}}{n!}\cdot
	\sum_{n=1}^\infty (2n-1)  5^{n-1} F_j^{2n-2}B_{2n} \frac{z^{2n-2}}{(2n)!}\,.
	\end{equation*}
	
	In the series $S_1(z)$ the first two terms are zero and therefore
	\begin{equation*}
	S_1(z) = \frac{1}{5F_j^2} \sum_{n=0}^\infty \frac{2^{n+2} L_{j(n+2)} - 2 L_j^{n+2}}{(n+2)(n+1)} \frac{z^{n}}{n!}\,.
	\end{equation*}
	
	The second series $S_2(z)$ equals
	\begin{equation*}
	S_2(z) = \sum_{n=0}^\infty \big ( 2^n L_{jn} - 2 L_j^n \big ) \frac{z^{n}}{n!}\cdot
	\sum_{n=0}^\infty (2n+1) 5^{n} F_j^{2n} B_{2n+2} \frac{z^{2n}}{(2n+2)!}
	\end{equation*}
	and is a simple Cauchy product. Expanding and comparing the coefficients of $z^n$ proves \eqref{thm1-1}. 
	
	Identity \eqref{thm1-2} follows from the functional equation \eqref{thm1-2} combined with \eqref{eqlem2} and
	\begin{equation*}
	\frac{1}{\cosh^2 z} = \frac{d}{dz} \tanh z = \sum_{n=1}^\infty (2n-1) 2^{2n} (2^{2n}-1) B_{2n} \frac{z^{2n-2}}{(2n)!}\,.
	\end{equation*}
	
	The underlying functional equation for identity \eqref{thm1-3} is
	\begin{equation}\label{fl_id}
	\frac{F(z)L(z)}{\sinh^2\!\Big ( \frac{\sqrt{5}F_j}{2}  z\Big )} = \frac{4}{\sqrt{5}} e^{L_j z} \coth\!\Big ( \frac{\sqrt{5}F_j}{2}  z\Big ).
	\end{equation}
	
	By \eqref{eqlem3}, \eqref{egF} and \eqref{egL}, the LHS of \eqref{fl_id} is 
	\begin{gather*}
	\frac{F(z)L(z)}{\sinh^2\!\Big ( \frac{\sqrt{5}F_j}{2}  z\Big )} = 
	\frac{4}{5F_j^2} \sum_{n=0}^\infty 2^n F_{jn}\frac{z^{n-2}}{n!} - 4\sum_{n=0}^\infty 2^n F_{jn}\frac{z^n}{n!}\cdot
	\sum_{n=0}^\infty \frac{B_{2n+2}5^{n} F_j^{2n}}{2n+2}  \frac{z^{2n}}{(2n)!} \\
	= \frac{8}{5F_j z} + 4\left(\frac{1}{5F_j^2} \sum_{n=0}^\infty 2^{n+2} \frac{F_{j(n+2)}}{(n+2)(n+1)}\frac{z^{n}}{n!}- \sum_{n=0}^\infty 2^n F_{jn}\frac{z^n}{n!}\cdot
	\sum_{n=0}^\infty \frac{B_{2n+2}}{2n+2} 5^{n} F_j^{2n} \frac{z^{2n}}{(2n)!}\right),
	\end{gather*}
	whereas the RHS of \eqref{fl_id} equals
	\begin{gather*}
	\frac{4}{\sqrt{5}} e^{L_j z} \coth \Big ( \frac{\sqrt{5}F_j}{2}  z\Big ) = \frac{4}{\sqrt{5}}\!\left ( \frac{2}{\sqrt{5}F_j} \sum_{n=0}^\infty L_{j}^n \frac{z^{n-1}}{n!} + \sum_{n=0}^\infty L_{j}^n \frac{z^n}{n!}
	\cdot 2 \sum_{n=1}^\infty B_{2n} 5^{\frac{2n-1}{2}} F_j^{2n-1} \frac{z^{2n-1}}{(2n)!}\right) \\
	=\frac{8}{5F_j z} + \frac{8}{\sqrt{5}} \left( \frac{1}{\sqrt{5}F_j} \sum_{n=0}^\infty \frac{L_{j}^{n+1}}{n+1}\frac{z^{n}}{n!}+\, z \sum_{n=0}^\infty L_{j}^n \frac{z^n}{n!}\cdot
	\sum_{n=0}^\infty \frac{B_{2n+2}}{(2n+2)(2n+1)} 5^{\frac{2n+1}{2}} F_j^{2n+1} \frac{z^{2n}}{(2n)!}\right).
	\end{gather*}
	Now, we can apply the Cauchy multiplication theorem on both sides. When simplifying the RHS, use $\binom{n-1}{k}\frac{n}{n-k} = \binom{n}{k}$. 
\end{proof}

Formula \eqref{thm1-1} has been derived recently in \cite{GuoChu}. 

When $j=1$, then the special cases of Theorem \ref{thm1} reduce to
\begin{equation}\label{FQproblem}
\sum_{\substack{k=0 \\ n - k \equiv 0 \,(\mathrm {mod} \,2)}}^n \binom{n}{k} 
\big ( 2^k L_{k} - 2 \big ) (\sqrt{5})^{n-k} \frac{B_{n-k+2}}{n-k+2} 
= \frac{2^{n+2}L_{n+2}-2}{{5}(n+1)(n+2)} - 1\,,
\end{equation}
\begin{equation*}
\sum_{\substack{k=0 \\ n - k \equiv 0 \,(\mathrm {mod} \,2)}}^n \binom{n}{k} 
\big ( 2^k L_{k} + 2 \big ) (\sqrt{5})^{n-k} \frac{2^{n-k+2}-1}{n-k+2} B_{n-k+2} = 1
\end{equation*}
and
\begin{align*}
\sum_{\substack{k=0 \\ n - k \equiv 0 \,(\mathrm {mod} \,2)}}^n \binom{n}{k} 2^k F_{k} (\sqrt{5})^{n-k} &\frac{B_{n-k+2}}{n-k+2} + \frac{2}{\sqrt{5}} \sum_{\substack{k=0 \\ n - k \equiv 1 \,(\mathrm {mod} \,2)}}^{n-1} \binom{n}{k} (\sqrt{5})^{n-k} 
\frac{B_{n-k+1}}{n-k+1}\\
& = \frac{2}{5(n+1)}\Big ( \frac{2^{n+1}F_{n+2}}{n+2} - 1 \Big ).
\end{align*}

The identity \eqref{FQproblem} appeared as a problem proposal in \cite{RF-FQ}. 
\begin{remark} Using the fact that, for any sequence $(a_n)_{n\geq 0}$,
	\begin{equation*}
	\sum_{\substack{k=0 \\ n - k \equiv 0 \,(\mathrm {mod} \,2)}}^n a_k = \sum_{k=0}^n \frac{1+(-1)^{n-k}}{2} a_k,\qquad
	\sum_{\substack{k=0 \\ n - k \equiv 1 \,(\mathrm {mod} \,2)}}^n a_k = \sum_{k=0}^n \frac{1-(-1)^{n-k}}{2} a_k\,,
	\end{equation*}
	identities \eqref{thm1-1}--\eqref{thm1-3} possess the following equivalent forms without the mod-notation {\rm(}with $n$ even{\rm )}:
	\begin{equation*}
	\sum_{k=0}^{n/2} \binom{n}{2k} \frac{n-2k-1}{(k+1)(2k+1)} \cdot\frac{L_j^{2k+2}-2^{2k+1}L_{2j(k+1)}}{(5F_j^2)^{k+1}}\, B_{n-2k} 
	= \Big (\frac{L_j}{\sqrt{5}F_j}\Big )^n,
	\end{equation*}
	\begin{equation*}
	\sum_{k=0}^{n/2} \binom{n}{2k} \frac{2^{n-2k+2}-1}{n-2k+2} \cdot\frac{2L_j^{2k}+2^{2k}L_{2jk}}{(5F_j^2)^{k}}\, B_{n-2k+2} 
	= \Big (\frac{L_j}{\sqrt{5}F_j}\Big )^n,
	\end{equation*}
	\begin{equation*}
	\sum_{k=0}^{n/2} \binom{n}{2k} \Big (\frac{4}{5F_j^2}\Big )^k \Big (\frac{n-2k-1}{2k+1} F_{j(2k+1)} + \frac{F_j L_j^{2k}}{4^k}\Big )\, B_{n-2k} = 0\,.
	\end{equation*}
\end{remark}
\begin{theorem}\label{thm2}
	For all positive integers $n$ and $j$,
	\begin{equation}\label{main4}
	\sum_{k=0}^n \binom{n}{k} \frac{2^k}{k+1} \left((-1)^k\frac{F_{j(k+1)}}{L^k_j}\Bigl( \frac{L_j}{\sqrt5F_j}\Bigr)^n 
	- (1+(-1)^n) \frac{2^{k+3}-2}{k+2} F_j B_{k+2}\right)=0\,.
	\end{equation}
	In particular,
	\begin{equation*}
	\sum\limits_{k = 1}^{2n} {( - 1)^k \binom{2n - 1}{k - 1}\frac{{2^kF_{jk} }}{{kL_j^k }} }  = 0\,.
	\end{equation*}
	Similarly, 
	\begin{equation}\label{main5}
	\sum_{k=0}^{n}\binom{n}{k}2^k\left((-1)^k\frac{L_{jk}}{L^k_j}\Bigl(\frac{L_j}{\sqrt5F_j}\Bigr)^n-\frac{1+(-1)^n}{n-k+1}B_k\right) = 1+(-1)^n\,,
	\end{equation}
	\begin{equation*}
	\sum\limits_{k = 0}^{2n - 1} {( - 1)^k \binom{2n - 1}k\frac{{2^kL_{jk} }}{{L_j^k }} } = 0\,.
	\end{equation*}
\end{theorem}
\begin{proof} From the basic identity $\sinh z = \tanh z\cdot \cosh z$
	we get
	\begin{align}\label{Formula1-main}
	\frac{\sqrt5}{2}F(z)e^{-\frac{L_j}{2}z} = \tanh \Big(\frac{\sqrt5 F_j}{2}z\Big) \cosh \Big(\frac{\sqrt5 F_j}{2}z\Big).
	\end{align}
	Then
	\begin{align*}
	\text{LHS of \eqref{Formula1-main}}&=
	\frac{\sqrt5}{2}\sum_{n=1}^{\infty}F_{jn}\frac{z^n}{n!}\cdot 
	\sum_{n=0}^{\infty}\Big(-\frac{L_j}{2}\Big)^n\frac{z^n}{n!}\\
	&=
	\frac{\sqrt5}{2}\sum_{n=0}^{\infty}F_{j(n+1)}\frac{z^{n+1}}{(n+1)!} 
	\cdot\sum_{n=0}^{\infty}\Big(-\frac{L_j}{2}\Big)^n\frac{z^n}{n!}\\
	&=
	\frac{\sqrt5}{2}z\sum_{n=0}^{\infty}\sum_{k=0}^{n}\binom{n}{k}\frac{F_{j(k+1)}}{k+1} 
	\Big(-\frac{L_j}{2}\Big)^{n-k}\frac{z^n}{n!}\,,
	\end{align*}
	whereas
	\begin{align*}
	\text{RHS of \eqref{Formula1-main}}&=\sum_{n=1}^{\infty}2^{2n}(2^{2n}-1)\Big(\frac{\sqrt5 F_j}{2}\Big)^{2n-1}B_{2n}\frac{z^{2n-1}}{(2n)!}\cdot \sum_{n=0}^{\infty}\Big(\frac{\sqrt5 F_j}{2}\Big)^{2n}\frac{z^{2n}}{(2n)!}\\
	&=\sum_{n=0}^{\infty}2^{2n+2}(2^{2n+2}-1)\Big(\frac{\sqrt5 F_j}{2}\Big)^{2n+1}B_{2n+2}\frac{z^{2n+1}}{(2n+2)!}\cdot\sum_{n=0}^{\infty}\Big(\frac{\sqrt5 F_j}{2}\Big)^{2n}\frac{z^{2n}}{(2n)!}\\
	&=z\sum_{n=0}^{\infty}\sum_{k=0}^{n}\frac{1+(-1)^n}{2}\cdot \frac{1+(-1)^k}{2}\,2^{k+2}(2^{k+2}-1)\\
	&\,\quad\times\Big(\frac{\sqrt5 F_j}{2}\Big)^{k+1}\frac{B_{k+2}}{(k+2)!}\Big(\frac{\sqrt5 F_j}{2}\Big)^{n-k}\frac{z^{n }}{(n-k)!}\\
	&=z\,\Big(\frac{\sqrt5 F_j}{2}\Big)^{n+1}\sum_{n=0}^{\infty}\sum_{k=0}^{n}\frac{1+(-1)^n}{2}\cdot \frac{1+(-1)^k}{2}\\
	&\,\quad\times\binom{n}{k}2^{k+2}(2^{k+2}-1)\frac{B_{k+2}}{(k+2)(k+1)}\frac{z^{n}}{n!}\\
	&=z\,\Big(\frac{\sqrt5 F_j}{2}\Big)^{n+1}\sum_{n=0}^{\infty}\sum_{k=0}^{n}\frac{1+(-1)^n}{2}\binom{n}{k}2^{k+2}(2^{k+2}-1)
	\frac{B_{k+2}}{(k+2)(k+1)}\frac{z^{n}}{n!}.
	\end{align*}
	
	Note that above we used
	\begin{equation*}
	\sum_{n=0}^{\infty}a_{2n}z^{2n}\cdot \sum_{n=0}^{\infty}b_{2n}z^{2n} =\sum_{n=0}^{\infty}\sum_{k=0}^{n}\frac{1+(-1)^n}{2}\frac{1+(-1)^k}{2} a_{k}b_{n-k}z^n.
	\end{equation*}
	
	Comparing the coefficients of $z^n$ after some simple manipulations we have \eqref{main4}. 
	
	Identity \eqref{main5} follows from
	$\displaystyle\cosh z = \coth z\cdot \sinh z$ which gives
	\begin{equation}\label{Formula3-main}
	\frac{1}{2}L(z)e^{-\frac{L_j}{2}z} = {\coth \Big(\frac{\sqrt5 F_j}{2}z\Big)} \sinh \Big(\frac{\sqrt5 F_j}{2}z\Big).
	\end{equation}

	Proceeding as before,
	\begin{align*}
	\text{LHS of \eqref{Formula3-main}}&=\frac{1}{2}\sum_{n=0}^{\infty}L_{jn}\frac{z^n}{n!}\cdot\sum_{n=0}^{\infty}
	\left(-\frac{L_j}{2}\right)^n\frac{z^n}{n!}\\
	&=\frac{1}{2}\sum_{n=0}^{\infty}\sum_{k=0}^{n}\binom{n}{k}L_{jk}\left(-\frac{L_j}{2}\right)^{n-k}\frac{z^n}{n!},
	\end{align*}
	and
	\begin{align*}
	\text{RHS  of \eqref{Formula3-main}}&=\sum_{n=0}^{\infty}4^n \Big(\frac{\sqrt5 F_j}{2}\Big)^{2n-1}B_{2n}\frac{z^{2n-1}}{(2n)!}  
	\cdot\sum_{n=0}^{\infty} \Big(\frac{\sqrt5 F_j}{2}\Big)^{2n+1}\frac{z^{2n+1}}{(2n+1)!}\\
	&=\sum_{n=0}^{\infty}\sum_{k=0}^{n}\binom{n}{k}\frac{1+(-1)^n}{2}\frac{1+(-1)^k}{2}2^k   
	\Big(\frac{\sqrt5 F_j}{2}\Big)^{k-1}\frac{B_k}{n-k+1} \Big(\frac{\sqrt5 F_j}{2}\Big)^{n-k+1}\frac{z^{n}}{n!}\\
	&=\sum_{n=0}^{\infty}\sum_{k=0}^{n}\binom{n}{k}\frac{1+(-1)^n}{2}2^k 
	\Big(\frac{\sqrt5 F_j}{2}\Big)^{k-1}\frac{B_k}{n-k+1} \Big(\frac{\sqrt5 F_j}{2}\Big)^{n-k+1}\frac{z^{n}}{n!}.
	\end{align*}
	
	Comparing the coefficients of $z^n$ after some simple manipulations we have \eqref{main5}.
\end{proof}

In view of the binomial theorem and the Binet formula, \eqref{main5} is equivalent to
\begin{equation*}
\big(1 + (-1)^n \big)\sum_{k=0}^n \binom{n}{k} \frac{2^k B_k}{n - k + 1} = 0\,,
\end{equation*}
so that we have
\begin{equation*}
\sum_{k=0}^n \binom{n}{k} \frac{2^k B_k}{n - k + 1} = 0\,, \qquad\mbox{$n$ even}.
\end{equation*}
\begin{theorem}\label{thm3}
	For all positive integers $n$ and $j$,
	\begin{equation}\label{main6}
	\sum_{k=0}^{\lfloor{n}/{2}\rfloor}\binom{n}{2k} (5F_j^2)^k\left(\frac{F_{j(n-2k+1)}}{n-2k+1}{B_{2k}}-\frac{F_jL_j^{n-2k}}{2^n}\right) = 0\,,
	\end{equation}
	\begin{equation}\label{main7}
	\sum_{k=0}^{\lfloor{n}/{2}\rfloor}\binom{n}{2k} \frac{(5F^2_j)^k}{2k+1} \left(\frac{4^{k+1}-1}{k+1}L_{j(n-2k)}{B_{2k+2}}
	-\frac{L_j^{n-2k}}{2^n}\right) = 0\,.
	\end{equation}
\end{theorem}
\begin{proof}
	Use the exponential generating functions from \eqref{egF} and \eqref{egL} in conjunction with
	\begin{equation*}
	\sum_{n=0}^{\infty}a_{n}z^{n}\cdot \sum_{n=0}^{\infty}b_{2n}z^{2n}=\sum_{n=0}^{\infty}\sum_{k=0}^{\lfloor{n}/{2}\rfloor}a_{n-2k}b_{2k}z^n
	\end{equation*}
	and
	$\displaystyle\sinh z  = \frac{\cosh z}{\coth z}$, $\displaystyle\cosh z = \frac{\sinh z}{\tanh z}$.
\end{proof}

On account of the identity
\begin{equation*}
2\sum_{k = 0}^{\left\lfloor {n/2} \right\rfloor} \binom{n}{2k} x^{n-2k} z^{2k} = (x + z)^n + (x - z)^n\,,
\end{equation*}
formula \eqref{main6} can also be written as
\begin{equation*}
\sum_{k = 0}^{\left\lfloor {n/2} \right\rfloor} \binom{n}{2k} \frac{(5F_j^2)^kF_{j(n - 2k + 1)}}{n - 2k + 1}  B_{2k} = \frac{F_j L_{nj}}{2}.
\end{equation*}

\section{Special Bernoulli polynomial identities}

The properties of the Bernoulli polynomials stated in Lemmas \ref{lem.a4pymf5} and \ref{lem.j22kj78} below are direct consequences of the functional relations \eqref{1FuncRel-B} and \eqref{2FuncRel-B}.
\begin{lemma}\label{lem.a4pymf5}
	If $n$ is any non-negative integer, then
	\begin{equation*}
	B_n (1 + x) \pm B_n (1 + y) = B_n (x) \pm B_n (y) + n(x^{n - 1}  \pm y^{n - 1}),
	\end{equation*}
	\begin{equation*}
	B_n (1 + x) \pm B_n (1 + y) = (-1)^n\big(B_n (1 - x) \pm B_n (1 - y)\big) + n(x^{n - 1}  \pm y^{n - 1}),
	\end{equation*}
	\begin{equation}
	\label{eq.qfnqxj1}
	B_n ( - x) \pm B_n ( - y) =(-1)^n\big( B_n (x) \pm B_n (y) + n(x^{n - 1}  \pm y^{n - 1} )\big).
	\end{equation}
\end{lemma}
\begin{lemma}\label{lem.j22kj78}
	Let $n$ be any non-negative integer. If $x - y = 1$, then
	\begin{equation}\label{eq.u2rmh0q}
	B_n (x) - B_n (y) = ny^{n - 1},
	\end{equation}
	\begin{equation*}
	B_n ( - x) - B_n ( - y) = n(-1)^nx^{n - 1};
	\end{equation*}
	while if $x + y =1$, then
	\begin{equation}
	\label{eq.mtwbof7}
	B_n (x) - (-1)^n B_n (y)=0,
	\end{equation}
	\begin{displaymath}
	B_n (1 + x) - B_n (1 + y) =\begin{cases}
	n(x^{n - 1}  - y^{n - 1}), & \text{\rm if $n$ is even;} \\[6pt]
	- 2B_n (y) + n(x^{n - 1}  - y^{n - 1}), & \text{\rm otherwise.}
	\end{cases} 
	\end{displaymath}
\end{lemma}
\begin{lemma}\label{lem.nps28bv}
	For real or complex $z$, let a given well-behaved function $h(z)$ have in its domain the representation $h(z)=\sum\limits_{k=c_1}^{c_2}{v_kz^{w_k}}$, where $v_k$ and $w_k$ are given real sequences and
	\break $-\infty\leq c_1< c_2\leq+\infty$. Let $i$ and $m$ be integers. Then
	\begin{equation}\label{F}
	\sum_{k = c_1 }^{c_2 } {F_{iw_k + m} v_kz^{w_k} }  = 
	\frac{1}{\sqrt5}\big(\alpha^m h(\alpha ^i z) + \beta^m h(\beta ^i z)\big)\,,
	\end{equation}
	\begin{equation}\label{L}
	\sum_{k = c_1 }^{c_2 } {L_{iw_k + m} v_kz^{w_k}} = \alpha^m h(\alpha ^i z) - \beta^m h(\beta ^i z)\,.
	\end{equation}
\end{lemma}

Lemma \ref{lem.nps28bv} written in a slightly different form we can find in \cite[Theorem  1]{adegoke20x}. 
\begin{theorem}
	Let $j$ and $m$ be integers and $n$ a non-negative integer. Then
	\begin{equation}\label{Th4-1}
	\sum_{k = 0}^n {\binom nkF_{jk + m} B_{n - k} (x)z^k }  =\frac{1}{\sqrt5}\big( 
	\alpha^m B_n(x + \alpha ^j z) -\beta^m B_n (x + \beta ^j z)\big)\,,
	\end{equation}
	\begin{equation}\label{Th4-2}
	\sum_{k = 0}^n {\binom nkL_{jk + m} B_{n - k} (x)z^k }  = 
	\alpha^m B_n(x + \alpha ^j z) + \beta^m B_n (x + \beta ^j z).
	\end{equation}
\end{theorem}
\begin{proof} 
	Use \eqref{F} and \eqref{L} 
	with 
	$$
	h(z)=B_n (x + z) = \sum\limits_{k = 0}^n {\binom nkB_{n - k} (x)z^{k}},
	$$
so that $w_k=k$, $v_k=\binom nk B_{n - k} (x)$, $c_1=0$ and $c_2=n$.
\end{proof}

Setting $x=0$ in \eqref{Th4-1} and \eqref{Th4-2} yield the following Fibonacci--Bernoulli and Lucas--Bernoulli relations.
\begin{corollary}\label{cor.wxu7h9l}
	Let $j$ and $m$ be integers and $n$ non-negative integer. Then
	\begin{equation}\label{eq.xludrew1}
	\sum_{k = 0}^n {\binom nk F_{jk + m} B_{n - k} z^k }  =\frac{1}{\sqrt5}\big( 
	\alpha^m B_n(\alpha ^j z) - \beta^m B_n (\beta ^j z)\big)\,,
	\end{equation}
	\begin{equation}\label{eq.aneantm}
	\sum_{k = 0}^n {\binom nkL_{jk + m} B_{n - k} z^k }  = 
	\alpha^m B_n(\alpha ^j z) + \beta^m B_n (\beta ^j z).
	\end{equation}
\end{corollary}
\begin{theorem}
	Let $j$ and $m$ be integers and $n$  non-negative integer. Then
	\begin{equation}\label{eq.jik918k}
	\sum_{k = 0}^n {\binom nk \frac{F_{jk + m}}{L_j^{k}} B_{n - k} } = 
	\begin{cases}
	F_m B_n \Big(\frac{\alpha ^j }{L_j}\Big), & \text{\rm if $n$ is even;} \\[6pt]
	\frac{{L_m }}{{\sqrt 5 }}B_n\Big(\frac{\alpha ^j }{L_j}\Big), & \text{\rm  if $n$ is odd,}
	\end{cases}
	\end{equation}
	\begin{equation*}
	\sum_{k = 0}^n \binom nk\frac{L_{jk + m}}{L_j^{k}}B_{n - k} = 
	\begin{cases}
	{L_m} B_n \Big(\frac{\alpha ^j }{L_j}\Big), & \text{\rm  if $n$ is even;} \\[6pt]
	\sqrt5 F_m B_n \!\left(\frac{\alpha ^j }{L_j}\right), & \text{\rm  if $n$ is odd.}
	\end{cases}
	\end{equation*}
\end{theorem}
\begin{proof} Choose $x=\frac{\alpha^j}{L_j}$ in \eqref{2FuncRel-B} and use the Binet formula $L_j=\alpha^j+\beta^j$ to obtain
	\begin{equation}\label{alpha-beta}
	B_n \Big(\frac{\beta^j }{L_j}\Big) = (-1)^n B_n \Big(\frac{\alpha ^j }{L_j}\Big).
	\end{equation}
	
	Now use this information in Corollary~\ref{cor.wxu7h9l} with $z=\frac{1}{L_j}$.
\end{proof}
\begin{lemma}\label{lem.n4r4nrv}
	Let $a$, $b$, $c$ and $d$ be rational numbers and $\lambda$ an irrational number. Then $a + \lambda b = c + \lambda d$ if and only if $a = c$ and $b = d$.
\end{lemma}
\begin{corollary} Let $j$ be an integer and $n$ a non-negative integer. Then
	\begin{equation}\label{eq.mdgtnrw}
	\sum_{k = 0}^n \binom nk \frac{F_{jk - 1}}{L_j^k} B_{n-k} = B_n\Big(\frac{\alpha ^j }{L_j}\Big),\qquad\mbox{$n$ even},
	\end{equation}
	\begin{equation}\label{eq.bemj6wo}
	\sum_{k = 1}^n \binom nk \frac{F_{jk}}{L_j^k} B_{n-k} = 0,\qquad\mbox{$n$ even},
	\end{equation}
	\begin{equation}\label{eq.qzniemk}
	\sum_{k = 0}^n \binom nk  \frac{L_{jk-1}}{L_j^k} B_{n-k} =\sqrt5\, B_n\Big(\frac{\alpha ^j }{L_j}\Big),\qquad\mbox{$n$ odd},
	\end{equation}
	\begin{equation} \label{eq.i3xsg8e}
	\sum_{k = 0}^n \binom nk \frac{L_{jk}}{L_j^k} B_{n-k} = 0,\qquad\mbox{$n$ odd}.
	\end{equation}
\end{corollary}
\begin{proof} Since the expression on the left side of \eqref{eq.jik918k} is rational, being the finite sum of rational numbers, it follows that $B_n\big(\frac{\alpha ^j }{L_j}\big)$ is a rational number for even $n$. Now, using \eqref{eq.qr64rlf} and relation $\alpha^s=\alpha F_s+F_{s-1}$, we have
	\begin{align*}
	B_n\Big(\frac{\alpha ^j}{L_j}\Big)& = \sum_{k = 0}^n {\binom nk\frac{{B_k\alpha ^{j(n - k)} }}{{L_j^{n - k} }}}\\
	&= \alpha \sum_{k = 0}^n \binom nk\frac{B_k F_{j(n - k)}}{L_j^{n - k}}  + \sum_{k = 0}^n \binom nk\frac{B_k F_{j(n-k) - 1}}{L_j^{n - k}},
	\end{align*}
	from which identities \eqref{eq.mdgtnrw} and \eqref{eq.bemj6wo} follow when we invoke Lemma \ref{lem.n4r4nrv}. The proof of \eqref{eq.qzniemk} and \eqref{eq.i3xsg8e} is similar. 
\end{proof}
\begin{remark}\label{remark}
	We observe from identity \eqref{eq.jik918k} that $\sqrt5 B_n \big(\frac{\alpha ^j }{L_j}\big)$ is rational for $n$ odd.
\end{remark}
\begin{theorem}
	Let $j$ and $m$ be integers and $n$ non-negative integer. Then
	\begin{equation*}
	\sum_{k = 0}^n {( - 1)^k \binom nk \frac{F_{jk + m}}{L_j^{k}} B_{n - k}} = \begin{cases}
	F_m B_n \big(\frac{\alpha ^j }{L_j }\big) + \frac{nF_{j(n - 1) + m}}{L_j^{n-1}}, & \text{\rm if $n$ is even;} \\[6pt]
	- \frac{{L_m }}{{\sqrt 5 }}B_n \big(\frac{\alpha ^j }	{L_j }\big) - \frac{nF_{j(n - 1) + m}}{L_j^{n-1}}, & \text{\rm otherwise.}
	\end{cases}
	\end{equation*}
	\begin{equation*}
	\sum_{k = 0}^n {( - 1)^k \binom nk \frac{L_{jk + m}}{L_j^{k}} B_{n - k}} =
	\begin{cases}
	L_m B_n \big(\frac{\alpha ^j }{L_j }\big) + \frac{nF_{j(n - 1) + m}}{L_j^{n-1}}, & \text{\rm if $n$ is even;} \\[6pt]
	- \sqrt 5 {F_m }B_n \big(\frac{\alpha ^j }{L_j }\big) - \frac{nL_{j(n - 1) + m}}{L_j^{n-1}}, & \text{\rm otherwise,}
	\end{cases}
	\end{equation*}
\end{theorem}
\begin{proof} Identities  \eqref{3FuncRel-B} and \eqref{alpha-beta} 
	give
	\begin{equation*}
	B_n \Big(\!-\frac{\alpha ^j }{L_j }\Big) + B_n \Big(\!-\frac{ \beta ^j }{L_j}\Big) = \begin{cases}
	2B_n \Big(\frac{\alpha ^j }{L_j }\Big) + \frac{nL_{j(n - 1)}}{L_j^{n - 1}}, & \text{\rm if $n$ is even;} \\[6pt]
	- \frac{nL_{j(n - 1)}}{L_j^{n - 1}}, & \text{otherwise,}
	\end{cases}
	\end{equation*}
	\begin{equation*}
	B_n \Big(\!-\frac{\alpha ^j }{L_j }\Big) - B_n \Big(\! -\frac{ \beta^j }{L_j}\Big) = 
	\begin{cases}
	\frac{\sqrt 5nF_{j(n - 1)}}{L_j^{n - 1}}, & \text{if $n$ is even;} \\[6pt]
	- 2B_n \Big(\frac{\alpha ^j }{L_j }\Big) - \frac{ \sqrt 5nF_{j(n - 1)}}{L_j^{n - 1}}, & \text{\rm otherwise.}
	\end{cases}
	\end{equation*}
	
	Use these in Corollary~\ref{cor.wxu7h9l} with $z=-\frac{1}{L_j}$. Note the use of the known identities \cite[Vol.~1, p.~111]{Koshy}   
	\begin{equation*}
	F_r L_s  + F_s L_r  = 2F_{r + s},\qquad L_r L_s  + 5F_s F_r  = 2L_{r + s}.
	\end{equation*}  
\end{proof}
\begin{theorem}\label{thm.hjxt955}
	Let $j$ be integer and $n$ non-negative integer. Then
	\begin{equation}\label{eq.v7yjuc8}
	\sum_{k= 0}^n \binom nk   \frac{2^kF_{jk}}{L_j^{k}} B_{n - k} = \frac{n}{\sqrt5}\Big(\frac{\sqrt5 F_j}{L_j}\Big)^{n - 1},\qquad\mbox{$n$ even},
	\end{equation}
	\begin{equation}\label{eq.pxmrr6j}
	\sum_{k = 0}^n \binom nk   \frac{2^kL_{jk}}{L_j^{k}} B_{n - k} = n\Big(\frac{\sqrt5 F_j}{L_j}\Big)^{n - 1},\qquad\mbox{$n$ odd}.
	\end{equation}
\end{theorem}
\begin{proof} Setting $x=\frac{2\alpha^j}{L_j} - 1$  
	in \eqref{1FuncRel-B} yields 
	\begin{equation}\label{eq.b26z27o}
	B_n \Big(\frac{2\alpha ^j}{L_j}\Big) - 
	B_n \Big(\frac{F_j \sqrt 5}{L_j}\Big) = n\Bigl( {\frac{{F_j \sqrt 5 }}{{L_j }}} \Bigr)^{n - 1}.
	\end{equation}
	Setting $x=\frac{2\beta^j}{L_j}$ in \eqref{2FuncRel-B} gives
	\begin{equation}\label{eq.o4jl58o}
	B_n \Big(\frac{2\beta ^j }{L_j}\Big) - B_n \Big(\frac{F_j \sqrt 5}{ L_j}\Big) = 0,\qquad\mbox{$n$ even}.
	\end{equation}
	From \eqref{eq.b26z27o} and \eqref{eq.o4jl58o} we find
	\begin{equation*}
	B_n \Big(\frac{2\alpha ^j } {L_j }\Big) - B_n \Big(\frac{2\beta ^j} { L_j}\Big) = n\Bigl( {\frac{{F_j \sqrt 5 }}{{L_j }}} \Bigr)^{n - 1},\qquad\mbox{$n$ even},
	\end{equation*}
	from which, upon use in \eqref{eq.xludrew1}, with $m=0$ and $z=\frac{2}{L_j}$, identity \eqref{eq.v7yjuc8} follows.
	
	Using $x=\frac{2\beta^j}{L_j}$ in \eqref{2FuncRel-B} gives
	\begin{equation}\label{eq.bt3cihl}
	B_n \Big(\frac{2\beta ^j }{L_j }\Big) + B_n \Big(\frac{F_j \sqrt 5 } {L_j }\Big) = 0,\qquad\mbox{$n$ odd}.
	\end{equation}
	Addition of~\eqref{eq.b26z27o} and~\eqref{eq.bt3cihl} produces
	\begin{equation*}
	B_n \Big(\frac{2\alpha ^j }{L_j }\Big) + B_n \Big(\frac{2\beta ^j }	 {L_j }\Big) = n\Bigl( {\frac{{F_j \sqrt 5 }}{L_j }} \Bigr)^{n - 1},\qquad\mbox{$n$ odd},
	\end{equation*}
	from which, upon use in \eqref{eq.aneantm}, with $m=0$ and $z=\frac{2}{L_j}$, identity \eqref{eq.pxmrr6j} follows.
\end{proof}

The result of the next theorem exhibits strong similarity to the polynomial identi\-ties from Introduction.
\begin{theorem} \label{thm11} 
	The following identity is valid for all $n\geq 0$, $j\geq 1$, and complex $x$:
	\begin{equation}\label{Th8}
	\begin{split}
	\sum_{k=0}^n {n\choose k} 2^k F_{jk} (\pm \sqrt{5}F_j)^{n-k} &B_{n-k}(x) \\
	= n & F_j \Big ((\pm\sqrt{5}F_j x  + L_{j} )^{n-1}
	+ (\pm\sqrt{5}F_j (x-1)  + L_{j} )^{n-1}\Big ).
	\end{split}
	\end{equation}
\end{theorem}
\begin{proof} Since
	\begin{equation*}
	H(x,\sqrt{5} F_j z) = \frac{\sqrt{5} F_je^{\frac{\sqrt5F_j}{2} (2x-1)z}}{2\sinh\!\big ( \frac{\sqrt{5}F_j}{2} z\big )}z,
	\end{equation*}
	we get the relation
	\begin{equation*}
	F(z) L(z) H(x,\sqrt{5} F_j z) = 2 F_j z e^{L_j z} e^{\frac{\sqrt5F_j}{2}(2x-1)z} \cosh\Big ( \frac{\sqrt{5}F_j}{2} z\Big ).
	\end{equation*}
	Hence,  
	\begin{align*}
	\sum_{n=0}^\infty \sum_{k=0}^n {n\choose k}  2^k F_{jk} (\sqrt{5}F_j)^{n-k} B_{n-k}(x) \frac{z^n}{n!}& = F_j z e^{\big(\frac{\sqrt{5}F_j}{2}(2x-1) + L_j\big) z}\Big (e^{\frac{\sqrt{5}F_j}{2} z} + e^{-\frac{\sqrt{5}F_j}{2}z}\Big ) \\
	& = F_j z \Big (e^{(\sqrt{5}F_j x + L_j) z} + e^{(\sqrt{5}F_j(x-1) + L_j)z}\Big ).
	\end{align*}
	This proves \eqref{Th8} with the positive root. The second follows upon replacing $x$ by $1-x$ and using \eqref{2FuncRel-B}.
\end{proof}

Setting $x=0$ in Theorem \ref{thm2}, we have the following.
\begin{corollary}\label{cor21} 	For $n\geq 0$ and $j\geq 1$, 
	\begin{equation*}\label{eqcor21}
	\sum_{k=0}^n {n\choose k} 2^k F_{jk} (\pm \sqrt{5}F_j)^{n-k} B_{n-k} = n F_j \Big (L_{j}^{n-1} + (\mp\sqrt{5}F_j + L_{j} )^{n-1}\Big ).
	\end{equation*}
\end{corollary}
\begin{corollary}
	For $n\geq 0$ and $j\geq 1$, 
	\begin{align*}
	\sum_{k=0}^n {n\choose k} 2^k F_{jk} (\sqrt{5}F_j)^{n-k} B_{n-k}(\alpha) n F_j 2^{1-n} \Big ( (\sqrt{5} F_j + L_{j+3} )^{n-1} + (-\sqrt{5}F_j + L_{j+3} )^{n-1}\Big ),
	\end{align*}
	where $\alpha$ is the golden ratio. Also, for $j\geq 3$, we have the analog identity
	\begin{equation}
	\begin{split}
	\label{eq2cor_22}
	\sum_{k=0}^n {n\choose k} 2^k F_{jk} (-\sqrt{5}F_j)^{n-k}& B_{n-k}(\alpha)\\ = 
	n F_j 2^{1-n} &\Big ( (\sqrt{5} F_j - L_{j-3})^{n-1} + (-\sqrt{5}F_j - L_{j-3} )^{n-1}\Big ).
	\end{split}
	\end{equation}
\end{corollary}
\begin{proof} Set $x=\alpha$ in Theorem \ref{thm2} and simplify using $5F_n=L_{n+1}+L_{n-1}$.
\end{proof}

We mention the special case of \eqref{eq2cor_22} for $j=3$:
\begin{equation*}
\sum_{k=0}^n {n\choose k}  (-\sqrt{5})^{n-k} F_{3k}B_{n-k}(\alpha) = (-1)^{n-1}n L_{n-1}.
\end{equation*}

Also, inserting $\beta$ in \eqref{Th8} and setting $j=1$ we can state the identity
\begin{equation*}
\sum_{k=0}^n {n\choose k} 2^k F_{k} (\sqrt{5})^{n-k} B_{n-k}(\beta) = (-1)^{n-1}n L_{2n-2}.
\end{equation*}
\begin{corollary}
	Let $n$, $j$ and $q$ be integers with $n, j\geq 1$ and $q\geq 2$. Then
	\begin{align*}
	\sum_{k=0}^{n} {n\choose k} 2^k F_{jk}&  (\pm \sqrt{5}F_j )^{n-k} \big ( q^{1-(n-k)} - 1 \big ) B_{n-k}\\
	= n F_j& q^{1-n} \sum_{r=1}^{q-1} \Big ( (\pm\sqrt{5} F_j r + q L_j )^{n-1} + (\pm\sqrt{5} F_j (r-q) + q L_j )^{n-1}\Big ).
	\end{align*}
\end{corollary}
\begin{proof}  Formula \eqref{Raabe-B} gives
	\begin{equation*}
	\big ( q^{1-n} - 1 \big ) B_{n} = \sum_{r=1}^{q-1} B_n \Big (\frac{r}{q}\Big ).
	\end{equation*}
	Therefore, we can write
	\begin{align*}
	\sum_{k=0}^{n} {n\choose k}  2^k F_{jk} & (\pm \sqrt{5}F_j)^{n-k} \big ( q^{1-(n-k)} - 1 \big ) B_{n-k} \\
	& = n F_j \sum_{r=1}^{q-1} \Big ( \big (\pm\sqrt{5}F_j\,\frac{r}{q} + L_j \big )^{n-1} 
	+ \big (\pm\sqrt{5}F_j\,\big (\frac{r}{q}-1\big ) + L_j \big )^{n-1} \Big ) \\
	& = n F_j q^{1-n} \sum_{r=1}^{q-1} \Big ( (\pm\sqrt{5} F_j r + q L_j )^{n-1} + (\pm\sqrt{5} F_j (r-q) + q L_j )^{n-1}\Big ).
	\end{align*}
\end{proof}

We proceed with some examples. The special case $q=2$ takes the form
\begin{align*}
\sum_{k=0}^{n} {n\choose k} 2^k F_{jk} (\sqrt{5}F_j)^{n-k} \big ( 2^{1-(n-k)} - 1 \big ) B_{n-k} & = n F_j 2^{1-n} \Big ( (L_j + 2\alpha^j)^{n-1} + (L_j + 2\beta^j )^{n-1}\Big ) \\
& = n F_j 2^{1-n} \sum_{m=0}^{n-1} {n-1\choose m} 2^m L_{jm} L_j^{n-1-m}.
\end{align*}

For $j=1$ the left-hand side can be expressed in closed-form and we obtain after some manipulations
\begin{equation*}
\sum_{k=0}^{n} {n\choose k}  \Bigl(\frac{\sqrt5}{4}\Bigr)^{k} \big ( 2-2^k \big ) F_{n-k} B_{k} = \frac{nL_{3(n-1)}}{2^{2n-1}}.
\end{equation*}

Similarly the case $q=3$ is treated. The calculations are lengthy and omitted. The result is
\begin{align*}
\sum_{k=0}^{n} {n\choose k} 6^k  \big (1- 3^{n-k-1} \big ) F_{jk} (\sqrt{5}F_j)^{n-k}B_{n-k}= n F_j \sum_{m=0}^{n-1}x {n-1\choose m} \big ( 2^{n-1} 
+  4^m  \big )L_j^{n-1-m}L_{jm}.
\end{align*}

For $j=1$ we get
\begin{align*} 
\sum_{k=0}^{n} {n\choose k}  6^k  (\sqrt{5})^{n-k} \big ( {1-3^{n-k-1}}\big )F_{k} B_{n-k}
= n 2^{n-1}  L_{2n-2} + \sum_{m=1}^{n} {n\choose m} m4^{m-1} L_{m-1}.
\end{align*}

\section{Conclusion}
In this paper, we have discovered new identities relating Bernoulli polynomials  (numbers) to Fibonacci and Lucas numbers. In our future papers, we will discuss the analogue results for Euler polynomials (numbers) and Fibonacci and Lucas numbers as well as identities connecting Bernoulli polynomials (numbers) with Jacobsthal, Pell and balancing  numbers.

\bigskip
\hrule
\bigskip

\noindent 2020 {\it Mathematics Subject Classification}: Primary 11B68; Secondary 11B39, 05A15.
\bigskip

\noindent \emph{Keywords:} Bernoulli numbers and polynomials,	Fibonacci sequence, Lucas sequence, recurrence, generating function.

\end{document}